\documentclass{ifacconf}

\usepackage{graphicx}      
\usepackage{natbib}        
\usepackage{epstopdf}
\usepackage{amsmath,amssymb,amsfonts}
\usepackage{graphicx}

\usepackage{algorithm}
\usepackage{algorithmic}
\let\theoremstyle\relax   
\usepackage{amsthm}
\usepackage{subcaption}
\usepackage{tabularx}
\usepackage{multirow} 
\usepackage{verbatim}
\usepackage{booktabs}

\newtheorem{assumption}{Assumption}

\newtheorem{theorem}{Theorem}
\newtheorem{remark}{Remark}
\usepackage{mathrsfs}
\usepackage{xcolor}
\usepackage[acronym]{glossaries}
\makeglossaries
\newacronym{ADP}{ADP}{Adaptive Dynamic Programming}
\newacronym{RL}{RL}{Reinforcement Learning}
\newacronym{ESO}{ESO}{Extended State Observer}
\newacronym{ETM}{ETM}{Event-Triggered Mechanism}
\newacronym{ETC}{ETC}{Event-Triggered Control}
\newacronym{HJB}{HJB}{Hamilton–Jacobi–Bellman}
\newacronym{ZOH}{ZOH}{Zero-Order Hold}
\newacronym{UBS}{UBS}{Ultimately Bounded and Stable}
\newacronym{ADRC}{ADRC}{Active Disturbance Rejection Control}
\newacronym{BIBS}{BIBS}{Bounded-Input Bounded-State}
\newacronym{IBE}{IBE}{Instantaneous Bellman Error}
\newacronym{EBE}{EBE}{Extrapolated Bellman Error}
\newacronym{ITES}{ITES}{Integral of Time-weighted Euclidean norm of State}
\newacronym{BE}{BE}{Bellman Error}
\newacronym{ET}{ET}{Event Trigger}
\newacronym{LS}{LS}{Least Square}
\newacronym{MIMO}{MIMO}{Multi Input Multi Output}
\newacronym{PE}{PE}{Persistence of Excitation}
\newacronym{UUB}{UUB}{Uniform Ultimate Boundedness}
\newacronym{MIET}{MIET}{Minimum Inter-Event Time}

\begin{document}
\begin{frontmatter}

\title{Deep Reinforcement Learning Optimization for Uncertain Nonlinear Systems via Event-Triggered Robust Adaptive Dynamic Programming} 


\author[First]{Ningwei Bai} \and
\author[First]{Chi Pui Chan} \and
\author[First]{Qichen Yin} \and
\author[Second]{Tengyang Gong} \and
\author[First]{Yunda Yan} \and
\author[First]{Zezhi Tang}

\address[First]{Department of Computer Science, University College London, Gower Street, London, WC1E 6BT United Kingdom (e-mail: zezhi.tang@ucl.ac.uk).}
\address[Second]{Department of Electrical and Electronic Engineering, University of Manchester, Oxford Rd, Manchester, M13 9PL United Kingdom}

\begin{abstract}
This work proposes a unified control architecture that couples a \gls{RL}-driven controller with a disturbance-rejection \gls{ESO}, complemented by an \gls{ETM} to limit unnecessary computations. The \gls{ESO} is utilized to estimate the system states and the lumped disturbance in real time, forming the foundation for effective disturbance compensation. To obtain near-optimal behavior without an accurate system description, a value-iteration-based \gls{ADP} method is adopted for policy approximation. The inclusion of the \gls{ETM} ensures that parameter updates of the learning module are executed only when the state deviation surpasses a predefined bound, thereby preventing excessive learning activity and substantially reducing computational load. A Lyapunov-oriented analysis is used to characterize the stability properties of the resulting closed-loop system. Numerical experiments further confirm that the developed approach maintains strong control performance and disturbance tolerance, while achieving a significant reduction in sampling and processing effort compared with standard time-triggered \gls{ADP} schemes.
\end{abstract}


\end{frontmatter}

\section{Introduction}

Learning-based methods have become a fundamental paradigm in modern engineering systems, enabling algorithms to improve performance through data-driven adaptation without relying solely on explicit mathematical models. Over the past decade, advances in machine learning, particularly in function approximation, optimization, and representation learning, have significantly expanded the capability of intelligent systems operating under uncertainty, compared to traditional analytical methods~(\cite{qinRealTimeRemainingUseful2023,zhangTabNetLocallyInterpretable2024,huHierarchicalTestingRabbit2025}). These approaches have been increasingly adopted in control, robotics, and even generative language models~(\cite{luDeterminationSurfaceCrack2020, zhaoRealTimeObjectDetection2024,tangTemporalScaleTransformer2025,yaoLEARNINGADOPTIONUNDERSTANDING2025}). However, conventional model-based techniques may be limited in their ability to handle nonlinearities, unknown disturbances, or incomplete system knowledge.

Reinforcement Learning (RL) has gained attention for complex decision-making and control in uncertain, dynamic environments~(\cite{tangReinforcementLearningBasedApproach2024}). In control engineering, \gls{RL}-based methods offer a data-driven alternative to classical model-based designs. This is useful when accurate system models are difficult to obtain. Among these methods, \gls{ADP} integrates \gls{RL} with optimal control theory. It facilitates near-optimal control of nonlinear systems by approximating value functions and control policies through function approximators. This eliminates the need to explicitly solve the \gls{HJB} equation~(\cite{lewisReinforcementLearningAdaptive2009}). 
However, conventional \gls{ADP} frameworks often rely on continuous or periodic updates to neural network parameters. These updates impose significant computational burdens and may lead to overfitting to transient disturbances or noise.

Event-triggered strategies have been widely adopted in diverse control applications, including networked and embedded systems, multi-agent coordination, and resource-constrained robotic platforms (\cite{onuohaStressMatrixBasedFormation2024,onuohaDiscreteTimeStressMatrixBased2024}). Meanwhile, the \gls{ETM} has been widely employed in both control and \gls{ADP} frameworks to reduce computational load~(\cite{hanEventtriggeredbasedOnlineIntegral2024a, heemelsIntroductionEventtriggeredSelftriggered2012, tabuadaEventTriggeredRealTimeScheduling2007}). Unlike time-driven schemes, \gls{ETM}s update only when systems meet a state- or error-based condition. State deviation or estimation error often directly triggers updates. This approach reduces redundant updates and preserves closed-loop stability~(\cite{dongEventTriggeredAdaptiveDynamic2017b,xueAdaptiveDynamicProgramming2020,onuohaStressMatrixBasedFormation2024}). By limiting updates to key events, event-triggered \gls{ADP} boosts efficiency and yields policies less sensitive to disturbances.

Despite these advantages, engineers must ensure robustness against external disturbances and modeling uncertainties. In practice, environmental perturbations, unmodeled dynamics, nonlinear couplings, and parameter uncertainties cause disturbances. Many robust control approaches employ feedback to reduce perturbations rather than explicitly use feedforward compensation~(\cite{tang2019control,tangUnmatchedDisturbanceRejection2016,tangDisturbanceObserverBasedOptimal2024}). In this context, a \gls{ESO} estimates the original states and accumulated interference in real time. This allows proactive compensation of parameter mismatches, unmodeled dynamics, and external perturbations in nonlinear systems~(\cite{luoAdaptiveAffineFormation2020, tangOutputTrackingUncertain2024, hanPIDActiveDisturbance2009, chenDisturbanceObserverBasedControlRelated2016, ranNewExtendedState2021c, puClassAdaptiveExtended2015, tangDisturbanceRejectionIterative2019}).


Recent work combines \gls{ESO}-based disturbance rejection with \gls{RL} for uncertain nonlinear systems~(\cite{ranReinforcementLearningBasedDisturbanceRejection2022,tangReinforcementLearningBasedOutput2024}). However, these \gls{ESO}–\gls{RL} schemes primarily operate in a time-driven manner: both the controller and learning updates run continuously or periodically, lacking an event-triggered learning mechanism. Many continuous-time \gls{ADP} designs also impose restrictive \gls{PE} conditions for parameter convergence~(\cite{kamalapurkarConcurrentLearningbasedApproximate2016}), making them hard to verify and enforce in practice.

Inspired by these observations, we develop a composite control framework for output-feedback control of uncertain nonlinear systems with lumped disturbances. The main contributions are summarized as follows:

\begin{enumerate}
    \item A unified control structure incorporating \gls{ETM} is developed, in which \gls{ESO}-based state estimation, disturbance compensation, and controller updates occur only at triggering instants. The resulting unified composite control framework enables an aperiodic and computationally efficient realization of output-feedback \gls{RL} control. By integrating \gls{ESO}-based disturbance rejection with event-triggered \gls{RL}, this work establishes a unified architecture not addressed in existing literature.
    \item We propose an event-triggered learning rule for a simulation-of-experience \gls{ADP} scheme, where critic and actor networks are updated solely at triggering instants using both instantaneous and \gls{EBE}. In contrast to existing \gls{ESO}–\gls{RL} frameworks~(\cite{ranReinforcementLearningBasedDisturbanceRejection2022,tangReinforcementLearningBasedOutput2024}), which employ continuous or periodic learning, the proposed mechanism produces an aperiodic, data-efficient adaptation. The associated analysis shows that practical stability and \gls{UUB} are achieved without imposing a classical \gls{PE} condition~(\cite{kamalapurkarConcurrentLearningbasedApproximate2016}), while avoiding redundant high-frequency learning updates.
\end{enumerate}

The remainder of the paper is structured as follows: Section II presents the problem and the system model.
Section III describes the proposed composite control framework.
Section IV describes the overall composite control framework.
The simulation results are demonstrated in Section V, and
Section VI summarizes this paper and outlines potential directions for future research.

\section{Problem Formulation}

In this paper, we identify the control of a set of uncertain affine nonlinear systems described by
\begin{equation}
\left\{\begin{array}{l}
\dot{z}=f_z(x, z, \eta), \\
\dot{x}=A x+B[f(x, z, \eta)+g(x, z, \eta) u], \\
y=C x,
\end{array}\right.
\label{equ:non_linear_system}
\end{equation}
where $x =\left[x_1, \dots, x_n\right]^{\mathrm{T}} \in \mathbb{R}^n$ denotes the state of the measured subsystem with relative degree $n$; $z \in \mathbb{R}^p$ represents the zero-dynamics state; $\eta \in \mathbb{R}$ denotes an external disturbance or uncertain parameter; $u \in \mathbb{R}$ is the control input; $f_z: \mathbb{R}^n \times \mathbb{R}^p \times \mathbb{R} \to \mathbb{R}^p$ is a smooth nonlinear mapping describing the evolution of zero dynamics; $f, g: \mathbb{R}^n \times \mathbb{R}^p \times \mathbb{R} \to \mathbb{R}$ are uncertain nonlinear functions characterizing the input and drift dynamics gain of the $x$ subsystem; and $A \in \mathbb{R}^{n \times n}$, $B\in \mathbb{R}^{n \times 1}$ and $C\in \mathbb{R}^{1 \times n}$ are the standard companion matrices defining a nominal chain-of-integrators structure of the output dynamics.

To enable subsequent observer and controller design, we impose the following standard assumptions.

\begin{assumption}
The external signal $\eta(t)$ as well as its time derivative $\dot\eta(t)$ are bounded for all $t \ge 0$.
\end{assumption}

\begin{assumption}
The zero dynamics $\dot{z}=f_z(x, z, \eta)$ with input $(x, \eta)$ is \gls{BIBS} stable.
\end{assumption}

In this paper, the nonlinear system dynamics with uncertainty are modeled as follows:
\begin{equation}\notag
\begin{aligned}
f(x, z, \eta) &=f_0(x)+\Delta f(x, z, \eta), \\
g(x, z, \eta) &=g_0(x)+\Delta g(x, z, \eta),
\end{aligned}
\end{equation}
where $f_0, g_0: \mathbb{R}^n \to \mathbb{R}$ denote the known nominal system dynamics; and $\Delta f, \Delta g: \mathbb{R}^n \times \mathbb{R}^p \times \mathbb{R} \to \mathbb{R}$ represent unknown disturbances and model uncertainties that may depend on the full state of the system $(x,z)$ and the external signal $\eta$. 

Following the \gls{ADRC} philosophy (\cite{hanPIDActiveDisturbance2009}), the general uncertainty is transferred to a broader state:
\begin{equation}
x_{n+1} \triangleq \Delta f(x, z, \eta)+\Delta g(x, z, \eta) u,
\end{equation}
According to this definition, the $n$-th subsystem dynamics can be rewritten as $\dot{x}_n=x_{n+1}+f_0(x)+g_0(x) u$, so that the overall system becomes an $(n+1)$th order augmented integrator chain perturbed by the unknown term $\dot{x}_{n+1}$.

To quantify performance, we consider the nominal compensated subsystem and assign the infinite-horizon cost functional
\begin{equation}\label{equ:cost_function}
J\left(x_0\right)=\int_0^{\infty}\left(Q(x(\tau))+u_0(\tau)^T R u_0(\tau)\right) d \tau,
\end{equation}
where $x_0=x(0)$ is the initial condition, $Q: \mathbb{R}^n \to \mathbb{R}_+$ is a positive definite state penalty, $R>0$ is a control-weighting matrix, and $u_0$ denotes the component of the input acting on the nominal dynamics after uncertainty compensation.

The associated optimal control problem is
\begin{equation}\notag
u_0^*=\arg \min _{u_0} J\left(x_0\right),
\end{equation}
and the optimal policy will be approximated online via the \gls{RL} mechanism developed later.


\begin{remark}
The purpose of this paper is to develop a \gls{ESO}-based \gls{RL} disturbance rejection scheme equipped with an \gls{ETM}. The proposed controller aims to stabilize the system under lumped uncertainties while achieving near-optimal performance with reduced control update frequency.
\end{remark}

\section{Composite Control Framework}

\subsection{\gls{ESO} Design}
First, a \gls{ESO} is designed to predict both the state of the system and all disturbances,
following the standard \gls{ADRC} structure~(\cite{hanPIDActiveDisturbance2009,chenDisturbanceObserverBasedControlRelated2016}):
\begin{equation}
\left\{\begin{array}{l}
\dot{\hat{x}}_i=\hat{x}_{i+1}+\frac{l_i}{\epsilon^i}\left(y-\hat{x}_1\right) \quad i=1, \ldots, n-1, \\
\dot{\hat{x}}_n=\hat{x}_{n+1}+\frac{l_n}{\epsilon^n}\left(y-\hat{x}_1\right)+f_0(\hat{x})+g_0(\hat{x}) u, \\
\dot{\hat{x}}_{n+1}=\frac{l_{n+1}}{\epsilon^{n+1}}\left(y-\hat{x}_1\right),
\end{array}\right.
\label{equ:ESO}
\end{equation}
where $\hat{x}=\left[\hat{x}_1, \ldots, \hat{x}_n, \hat{x}_{n+1}\right]^{\mathrm{T}}$, $\epsilon>0$ is a small positive constant adjusting the observer bandwidth, and $L=\left[l_1, \ldots, l_{n+1}\right]^{\mathrm{T}}$ is chosen that the following matrix is Hurwitz: 
\begin{equation}\notag
E=\left[\begin{array}{ccccc}
-l_1 & 1 & 0 & \cdots & 0 \\
-l_2 & 0 & 1 & \cdots & 0 \\
\vdots & \vdots & \vdots & \ddots & \vdots \\
-l_n & 0 & 0 & \cdots & 1 \\
-l_{n+1} & 0 & 0 & \cdots & 0
\end{array}\right] \in \mathbb{R}^{(n+1) \times(n+1)}.
\end{equation}

However, since the observer gains are scaled by $\epsilon^{-i}$, a small $\epsilon$ yields a high bandwidth \gls{ESO} that responds rapidly to state deviations but may induce a pronounced peaking phenomenon during the initial transient. 
To mitigate this effect, we employ a widely used smooth saturation technique to constrain the observer outputs. Let the saturated observer states be defined as
\begin{equation}\notag
\bar{x}_i=M_i s\left(\frac{\hat{x}_i}{M_i}\right), \quad i=1, \ldots, n+1,
\end{equation}
where $M_i>0$ are design bounds selected so that the saturation remains inactive during steady-state operation, and $s(\cdot)$ is an odd, continuously differentiable saturation-like function given by
\begin{equation}\notag
s(v)= \begin{cases}v, & 0 \leq v \leq 1 \\ v+\frac{v-1}{\varepsilon}-\frac{v^2-1}{2 \varepsilon}, & 1 \leq v \leq 1+\varepsilon \\ 1+\frac{\varepsilon}{2}, & v>1+\varepsilon,\end{cases}
\end{equation}
which satisfies $0 \leq s^{\prime}(v) \leq 1$, $|s(v)-\operatorname{sat}(v)| \leq \frac{\varepsilon}{2}, \forall v \in \mathbb{R}$. 
For later use, we denote $\bar{x}=\left[\bar{x}_1, \ldots, \bar{x}_{n+1}\right]^{\mathrm{T}} \in \mathbb{R}^{n+1}$ and observe that $\dot{\bar{x}}_i=s^{\prime}\left(\frac{\hat{x}_i}{M_i}\right) \dot{\hat{x}}_i. \quad i=1 \ldots, n+1$.

\subsection{\gls{ADP} Design}
Second, we present the \gls{ADP} design. An actor–critic architecture that is based on a neural network is employed, in which the critic approximates the optimal value function and the actor represents the corresponding optimal policy.

To facilitate the theoretical development of the \gls{ADP}-based controller and the associated optimized control law, the following standard assumptions are introduced:

\begin{assumption}
There exist constants $g_{\min }, g_{\max }>0$ such that
$$
g_{\min } \leq \inf _{x \in \mathcal{X}}\left|g_0(x)\right| \leq \sup _{x \in \mathcal{X}}\left|g_0(x)\right| \leq g_{\max }.
$$
Moreover, on $\Omega \triangleq \mathcal{X} \times \mathcal{Z} \times \mathcal{W}$ (where $\mathcal{W}$ is a compact set containing all admissible values of $\eta$ from Assumption~2, $\mathcal{X}$ is a compact set containing all admissible system states $x$, and $\mathcal{Z}$ is a bounded positively invariant set for the zero-dynamics state $z$), the relative mismatch between the true and nominal input gains is bounded
\begin{equation}\notag
\kappa_g \triangleq \sup _{(x, z, \eta) \in \Omega} \frac{\left|g(x, z, \eta)-g_0(x)\right|}{\left|g_0(x)\right|}<1.
\end{equation}
\end{assumption}


As shown in \eqref{equ:cost_function}, the associated value equation can be derived as follows
\begin{equation}\notag
V(x)=\min _{u} J(x)=\int_0^{\infty}\left(Q(x(\tau))+u(\tau)^T R u(\tau)\right) d \tau.
\end{equation}
The optimal value equation \(V(x)\) satisfies the \gls{HJB} equation
\begin{equation}\notag
0=\min _u\left[Q(x)+u^T R u+\nabla V(x)^T(f(x,z,\eta)+g(x,z,\eta) u)\right].
\end{equation}
gives the corresponding optimized control law 
\begin{equation}\label{equ:opt_law}
u^*(x)=-\frac{1}{2} R^{-1} g(x,z,\eta)^T \nabla V(x).
\end{equation}

Substituting \eqref{equ:opt_law} into the \gls{HJB} yields
\begin{equation}\notag
\begin{aligned}
0 &= Q(x) + \nabla V(x)^{T} f(x,z,\eta)\\
&\quad - \frac{1}{4}\, \nabla V(x)^{T} g(x,z,\eta) R^{-1} g(x,z,\eta)^{T} \nabla V(x).
\end{aligned}
\end{equation}

This expression provides the optimality condition for the value equation and forms the basis for the subsequent critic approximation in the \gls{ADP} framework.

\subsection{Training Process}

To avoid explicitly solving the HJB equation, we adopt an actor–critic architecture enhanced with an ESO-based extrapolation mechanism (\cite{ranReinforcementLearningBasedDisturbanceRejection2022}). The critic network approximates the value function using a linearly parameterized structure.

\begin{equation}\notag
V\left(\bar{x} ; W_c\right)=W_c^{\mathrm{T}} \phi(\bar{x}),
\end{equation}

where $W_c$ denotes the critic weight vector and $\phi(x)$ represents a basis vector.

The actor approximates the control policy as
\begin{equation}\notag
u_0\left(\bar{x} ; W_a\right)=-\frac{1}{2} R^{-1} g_0(\bar{x}) B^{\mathrm{T}} \phi_x^{\mathrm{T}}(\bar{x}) W_a,
\end{equation}
where $W_a$ is the actor weight vector and $\phi_x$ is the gradient of the basis function vector

Using the \gls{ESO}-estimated state $\hat{x}$, the \gls{IBE} is defined as
\begin{equation}\label{equ:IBE}
\begin{aligned}
\varepsilon_t \triangleq\; &
V_x(\bar{x}, W_a)\,
\Big[A\bar{x} 
+ B\big(f_0(\bar{x}) + g_0(\bar{x}) u_0(\bar{x}, W_a)\big)
\Big] \\[4pt]
&\quad +\, Q(\bar{x})
+ u_0^{T}(\bar{x}, W_c)\, R\, u_0(\bar{x}, W_c),
\end{aligned}
\end{equation}
which measures the deviation of the current actor–critic pair from the HJB optimality condition along the real trajectory. 

To enhance state-space coverage and improve robustness, an extrapolated dataset $\mathcal{X}_E=\left\{\xi^i\right\}_{i=1}^N$ is generated in the admissible domain, and the approximate \gls{EBE} is defined as follows:
\begin{equation}\label{equ:EBE}
\begin{aligned}
\varepsilon_i \triangleq\; &
V_x(\xi^i, W_a)\,
\Big[
A\xi^i 
+ B\big(f_0(\xi^i) + g_0(\xi^i) u_0(\xi^i, W_a)\big)
\Big] \\[4pt]
&\quad +\, Q(\xi^i)
+ u_0^{T}(\xi^i, W_c)\, R\, u_0(\xi^i, W_c).
\end{aligned}
\end{equation}
For least-squares type learning, we introduce the regressor
\begin{equation}\notag
\zeta = \phi(\bar{x}) \left[ A\bar{x} + B \bigl( f_0(\bar{x}) + g_0(\bar{x})\, u_0(\bar{x}, W_a) \bigr) \right],
\end{equation}
and the normalization term
\begin{equation}\notag
\sigma=1+\rho \zeta^{\mathrm{T}} \Psi \zeta,
\end{equation}
where $\rho>0$ is a constant normalization gain and the gain matrix $\Psi$ evolves by
\begin{equation}\notag
\dot{\Psi}=\left(\gamma \Psi-\alpha_{v 1} \frac{\Psi \zeta \zeta^{\mathrm{T}} \Psi}{\sigma^2}\right) \mathbf{1}_{\left\{\|\Phi\| \leq \delta_1\right\}}, \quad\|\Phi(0)\| \leq \delta_1,
\end{equation}
where $\gamma > 0$ is a constant forgetting factor, and the weight update for the critic is given by
\begin{equation}\label{equ:critic_update}
\dot{W_v} = -\alpha_{v1} \Psi \frac{\zeta}{\sigma}\varepsilon_i - \frac{\alpha_{v2}}{N} \Psi \sum_{i=1}^{N}{\frac{\zeta_i}{\sigma_i}\varepsilon_i},
\end{equation}
where $\delta_1>0$ is a saturation constant and constant $\alpha_{v1},\alpha_{v2}>0$ yielding a least-squares-type adaptation with improved convergence and numerical robustness.

The regressor and normalization term used for constructing the \gls{EBE} at the extrapolated sample points are defined as
\begin{equation}\notag
\begin{aligned}
\zeta_i &= \phi_x\left(\xi^i\right)\left[ A\xi^i 
+ B\left( f_0\left(\xi^i\right) + g_0\left(\xi^i\right) u_0\left(\xi^i, W_a\right) \right) \right],
\\[4pt]
\sigma_i &= 1 + \rho\, \zeta_i^{\mathrm{T}} \Psi \zeta_i,
\end{aligned}
\end{equation}
and the updated law for actor weight is
\begin{equation}\label{equ:actor_update}
\begin{aligned}
\dot{W}_a &= -\alpha_{c1}(W_a - W_v) - \alpha_{c2} W_a 
+ \frac{\alpha_{v1}\mathcal{H}_{t}^{\!T} W_a \zeta^{T}}{4\sigma}\, W_v \\[6pt]
&\quad + \sum_{i=1}^{N} 
\frac{\alpha_{v1}\mathcal{H}_{i}^{\!T} W_a \zeta_i^{T}}{4N\sigma_i}\, W_v,
\end{aligned}
\end{equation}\notag
where $\alpha_{c1},\alpha_{c2} > 0$ are constant adaptation gains and
\begin{equation}\notag
\begin{aligned}
& \mathcal{H}_t \triangleq \phi_x(\bar{x}) B g_0(\bar{x}) R^{-1} g_0^T(\bar{x}) B^T \phi_x^T(\bar{x}), \\
& \mathcal{H}_i \triangleq \phi_x\left(\xi^i\right) B g_0\left(\xi^i\right) R^{-1} g_0^T\left(\xi^i\right) B^T \phi_x^T\left(\xi^i\right).
\end{aligned}
\end{equation}
Finally, the resulting approximate optimal policy is expressed as
\begin{equation}\label{equ:u_update}
u=u_0(\bar{x},W_a)-\frac{\bar{x}_{n+1}}{g_0(\bar{x})}.
\end{equation}

\subsection{Event-Triggered Mechanism}


Let $\tau_k$ denote the $k$-th moment of triggering, and let $x(\tau_k)$ be the most recently transmitted state held by a \gls{ZOH} device. The event-triggering error is defined as
\begin{equation}\label{equ:et_error}
e(t) = x(\tau_k) - x(t), \quad t \in [\tau_k, \tau_{k+1}),
\end{equation}
that measures the mismatch between the last transmitted state and the current state. During each inter-event interval, the control policy and the neural-network parameters are held constant, i.e., $\pi(t) = \pi(x(\tau_k))$, and the actor–critic weights are frozen. When $\|e(t)\|$ exceeds a prescribed threshold, a new event is created, the state is transmitted, and both the control input and the network weights are updated.

To guarantee stability while avoiding unnecessary updates, the triggering threshold is chosen in a state-dependent and weight-adaptive form:
\begin{equation}\label{equ:threshold}
\delta(t)
= \sqrt{\frac{\lambda_{\min}(Q)\,\beta}{g_{max}^2 L_a^2 \,\|W_a(t)\|^2}} \; \|x(t)\|,
\end{equation}
that $\lambda_{\min}(Q)$ is the minimum eigenvalue of the cost matrix $Q$, $\beta \in (0,1)$ is a design constant introduced in the Lyapunov analysis, $g_{max}$ is a known upper bound of the nominal input gain $g_0(x)$, $L_a$ is the Lipschitz constant of the actor activation function, while $\|W_a(t)\|$ is an Euclidean norm of the actor weight vector. This construction links the triggering sensitivity to both the state magnitude and the current actor parameters: for large states or large weight norms, the condition becomes more stringent, enforcing timely updates; near the origin with well-converged weights, it becomes more relaxed, reducing redundant transmissions.

The triggering condition is explicitly given by 
\begin{equation}\label{equ:et_condition}
\|e(t)\|^2 > \delta(t)^2.
\end{equation}
An event is generated only when this inequality is satisfied. Under this rule, all of the closed-loop signals are uniformly bounded, while the number of updates is significantly reduced compared with a periodic scheme (\cite{dongEventTriggeredAdaptiveDynamic2017b}).

\section{Integrated Control Framework}

A unified control pipeline is depicted in Fig.~\ref{fig:pipeline}. At each control cycle, the \gls{ESO} updates the augmented state estimates, the event-triggered monitor tracks the deviation between the current and last transmitted states, and the \gls{ADP} either updates the actor–critic networks or holds the previous control input. A detailed procedure is summarized in Algorithm~\ref{alg:ETM-ESO-ADP}.

\begin{figure}[H]
\centering
\includegraphics[width=0.46\textwidth]{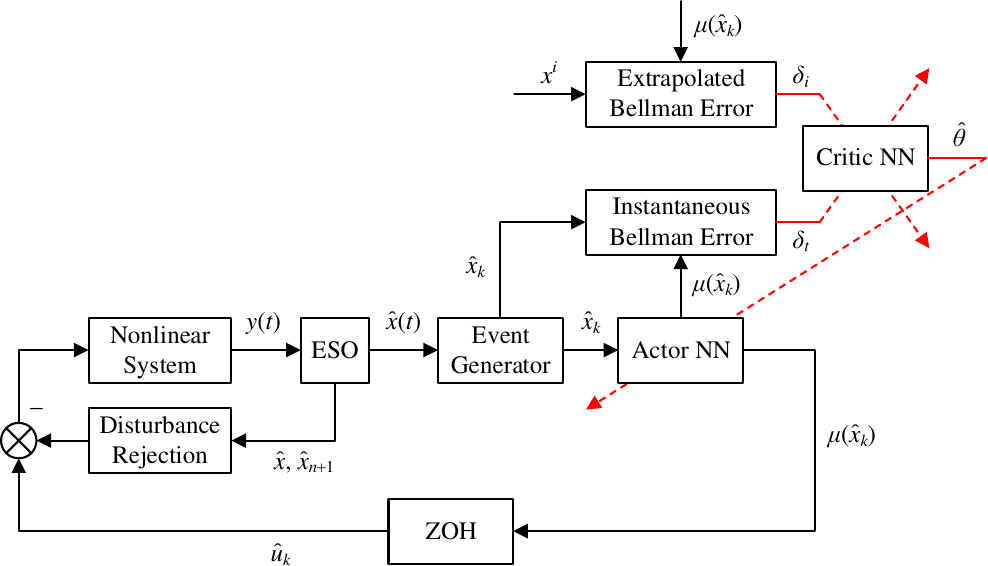}
\caption{Pipeline of the composite control framework.}
\label{fig:pipeline}
\end{figure}

\begin{algorithm}[ht!] 
\caption{Composite Control Framework}
\label{alg:ETM-ESO-ADP}
\begin{algorithmic}[1]

\STATE \textbf{Initialization}
\STATE Set critic and actor weights $W_c(0)$ and $W_a(0)$
\STATE Initialize \gls{ESO} states $z(0)$ in \eqref{equ:ESO}
\STATE Set initial triggering instant $\tau_0 = 0$ and store $x(\tau_0)$ and $u(\tau_0)$
\STATE Generate extrapolation set $\mathcal{X}_E$ for computing \gls{EBE}
\STATE \vspace{0.10cm}
\STATE Online Control Loop
\STATE At each time $t$:

\STATE \quad Update \gls{ESO} via \eqref{equ:ESO} to obtain $\hat{x}(t)$ and $\hat{x}_{n+1}(t)$
\STATE \quad Compute triggering error using \eqref{equ:et_error}
\STATE \quad Evaluate threshold $\delta(t)$ via \eqref{equ:threshold}
\STATE \quad Check triggering condition \eqref{equ:et_condition}

\STATE \quad If the triggering condition is satisfied then
\STATE \qquad Compute \gls{IBE} $\varepsilon_t$ using \eqref{equ:IBE}.
\STATE \qquad Compute \gls{EBE} $\varepsilon_i$ using \eqref{equ:EBE}
\STATE \qquad Update critic weights $W_v$ via \eqref{equ:critic_update}
\STATE \qquad Update actor weights $W_a$ via \eqref{equ:actor_update}
\STATE \qquad Compute new control input $u$ using \eqref{equ:u_update}
\STATE \qquad Set $\tau_{k+1}=t$ and store $\hat{x}(\tau_{k+1})$ and $u(\tau_{k+1})$

\STATE \quad Else

\STATE \qquad Hold $u(t)=u(\tau_k)$ via \gls{ZOH}.
\STATE \quad End if

\STATE \quad Apply $u(t)$ to the plant dynamics \eqref{equ:non_linear_system}.


\end{algorithmic}
\end{algorithm}

\begin{assumption}
Under the proposed \gls{ETM}, a \gls{MIET} $\tau_{\min}>0$ is enforced such that $\tau_{k+1}-\tau_k \ge \tau_{\min}, \forall k\in\mathbb{N}, $where $\tau_k$ denotes the sequence of triggering instants.
This condition prevents Zeno behavior and guarantees that the count of triggering events in any finite time interval remains finite.
\end{assumption}

\begin{theorem}\label{assumption4}
The composite control architecture includes the uncertain nonlinear system \eqref{equ:non_linear_system}, the \gls{ESO} \eqref{equ:ESO}, 
the \gls{ADP}-based controller \eqref{equ:u_update}, and the proposed \gls{ETM}. 
Suppose that:
(i) Assumptions~1–4 hold;
(ii) the observer gain vector is chosen such that the \gls{ESO} error dynamics 
are globally asymptotically stable;
(iii) the uncertainty compensation renders the plant dynamics equivalent to 
the nominal model used in the \gls{ADP} design; and
(iv) the estimation errors of the actor–critic neural-network weights are 
uniformly ultimately bounded.
Then all closed-loop signals are uniformly ultimately bounded, and the state $x(t)$ converges to a small neighborhood of the origin.
\end{theorem}

\begin{proof}
We provide a sketch of the argument. Consider the Lyapunov candidate
\begin{equation}\notag
V = V^{*}(x) + \frac{1}{2} (\Theta_v^{e})^{\mathrm{T}} \Gamma^{-1} \Theta_v^{e}
+ \frac{1}{2} (\Theta_c^{e})^{\mathrm{T}} \Theta_c^{e}.
\end{equation}
where \(V^{*}(x)\) is the optimal value function, and \(\Theta_v^{e}, \Theta_c^{e}\) denote the critic and actor weight errors, respectively.

On each inter-event interval \([\tau_k, \tau_{k+1})\), the \gls{ESO}, the control policy, and the actor–critic weights are fixed. Under the proposed control law, standard Lyapunov analysis yields
\begin{equation}\notag
\dot V \le -\alpha \|x\|^{2} + \mathcal{O}(\varepsilon),
\end{equation}
for some \(\alpha > 0\) and sufficiently small approximation and estimation errors, implying boundedness of all closed-loop variables on each inter-event interval. At each triggering instant, the \gls{ESO} correction and the actor–critic updates are designed so that \(V\) does not increase, which 
ensures \gls{UUB} of \((x, \Theta_v^{e}, \Theta_c^{e})\).

By Assumption~\ref{assumption4}, the inter-event times satisfy \(\tau_{k+1} - \tau_k \ge \tau_{\min} > 0\), which excludes Zeno behavior and guarantees that only finitely many events occur on any finite time interval.

Combining these properties shows that this composite control framework
system is uniformly ultimately bounded and that \(x(t)\) converges to a small 
neighborhood of the origin.
\end{proof}

\begin{remark}
The above Lyapunov analysis is carried out only on the continuous-time closed loop between triggering instants.
The event-triggering error $e(t)$ is not explicitly included in the Lyapunov function. As a result, the analysis provides practical stability within each inter-event interval, but does not constitute a full hybrid-system proof. A rigorous global stability proof would require augmenting the Lyapunov function with $e(t)$ and deriving an ISS-type inequality of the form

\begin{equation}
\dot V \le - k_1 \|x\|^2 + k_2 \|e\|^2,
\end{equation}

where $k_1>0$ is the nominal decay rate of the Lyapunov function, $k_2>0$ quantifies how the Lyapunov derivative is affected by the measurement error, with a triggering rule guaranteeing $\|e\| \le k_0 \|x\|$ and $k_2 k_0^2 < k_1$, while $k_0>0$ is the design parameter in the triggering condition that restricts the size of $e(t)$ relative to $x(t)$. A full stability analysis will be included in an extended journal version.

\end{remark}

\section{Simulation Results}
In what follows, two numerical case studies are conducted to evaluate the capability and robustness of the proposed composite control framework.
\subsection{Example 1}
To demonstrate the controller performance, we first consider a third-order uncertain nonlinear system. The real plant is given by
\begin{equation}
\left\{
\begin{aligned}
\dot z &= \underbrace{-\bigl(x_1^{2}+0.5\,\eta^{2}\bigr)\,z}_{f_z(x,z,\eta)}, \\[4pt]
\dot x_{1} &= x_{2}, \\[4pt]
\dot x_{2} &= 
\underbrace{-1.5\,x_{1}-x_{2}+1.5\,(x_{1}+x_{2})\bigl(\sin(x_{2})+2\bigr)^{2}}_{f_{0}(x)} \\[3pt]
&\quad + \underbrace{\bigl(-x_{2}+\eta+z^{2}\bigr)}_{\Delta f (x,z,\eta)} \\[3pt]
&\quad + \Bigl(
\underbrace{\cos(x_{1})+2}_{g_{0}(x)}
+\underbrace{\bigl(\sin(x_{2})-\eta\bigr)}_{\Delta g(x,z,\eta)}
\Bigr) u , \\[4pt]
y &= x_{1}.
\end{aligned}
\right.
\end{equation}


In this example, we construct the \gls{BE} over a uniformly discretized exploration set
\(\mathcal{X}_E\), defined as 
$
\mathcal{X}_E
=
[-2,\,2]_{0.5}
\;\times\;
[-2,\,2]_{0.5}
$
.The \gls{ESO} parameters are selected as follows. The observer gain vector is set to $
L = \begin{bmatrix}2 & 2 & 1\end{bmatrix}^{\mathrm{T}} $. The small positive constant is set to $\varepsilon = 0.03$ to accelerate \gls{ESO} convergence while maintaining robustness against measurement noise.

\noindent The saturation bounds for the \gls{ESO} outputs are chosen as $M_1 = M_2 = M_3 = 3,$. The nominal-function saturation limits are selected as $M_f = 7, M_g = 3$, ensuring boundedness of the \gls{ESO}-based control law.



\subsubsection{Performance Analysis}

Fig.~\ref{fig:state} demonstrates that the \gls{ESO} performs effectively for the considered nonlinear system. The observer-generated state trajectories closely track the actual system states, confirming the accuracy and reliability of the proposed observer.

\begin{figure}[ht!]
\centering
\includegraphics[width=0.45\textwidth]{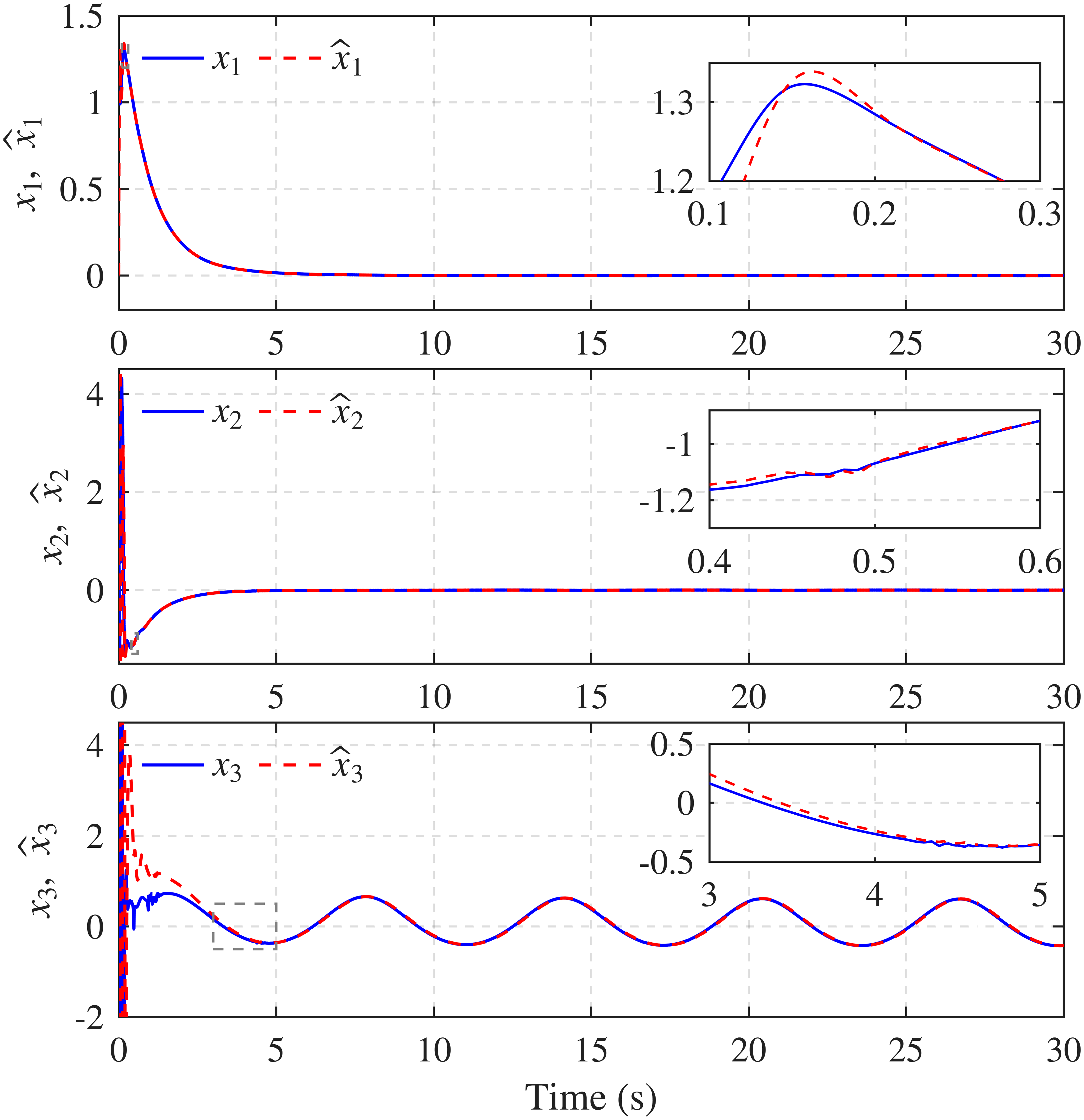}
\caption{System state trajectories and their \gls{ESO} estimates.}
\label{fig:state}
\end{figure}

Fig.~\ref{fig:ET_TT_states} shows that the control input $u$ under the proposed framework exhibits significant activity during the initial transient period due to uncertainties. After this short transient, the control input settles quickly and stays close to zero, maintaining stable behavior, indicating that the proposed control strategy achieves fast transient response and effective disturbance rejection.

To quantify the computational efficiency of the proposed \gls{ETM}, we define the update saving ratio as $ \eta = \frac{N_{\text{skipped}}}{N_{\text{total}}} \times 100\%, $where $N_{\text{skipped}}$ denotes the number of skipped control updates due to the triggering condition not being satisfied and $N_{\text{total}} = N_{\text{skipped}} + N_{\text{updated}}$ represents the total number of potential update instants. 
Based on the simulation study, an update-saving rate of $72\%$ is obtained, demonstrating a significant reduction in computational load while maintaining system stability.

\begin{figure}[ht!]
\centering
\includegraphics[width=0.48\textwidth]{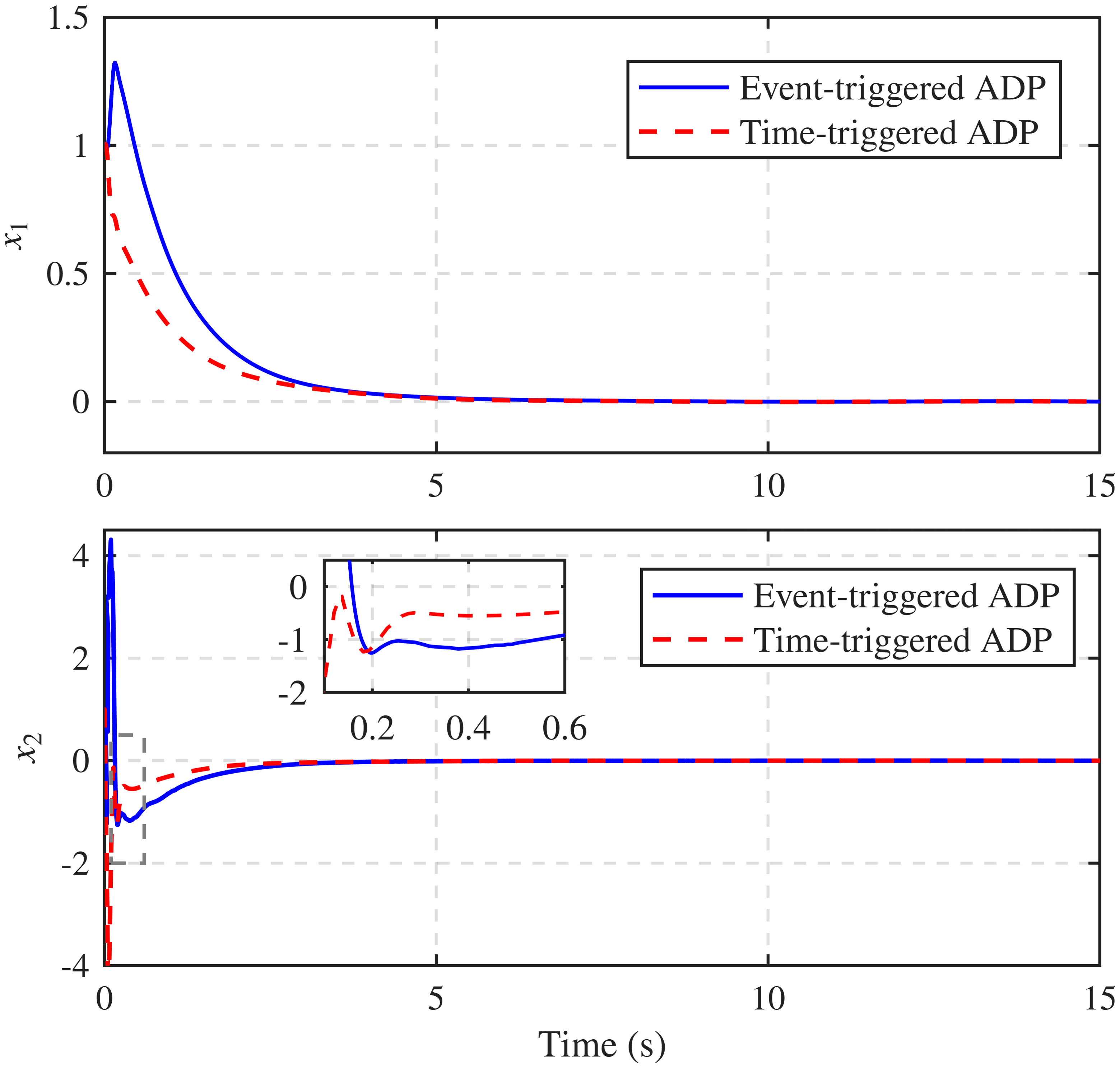}
\caption{State trajectories for composite control framework in Example 1}
\label{fig:ET_TT_states}
\end{figure}

As shown in Fig.~\ref{fig:ET_TT_states}, both $x_1$ and $x_2$ remain stable and converge to the equilibrium.

\begin{figure}[ht!]
\centering
\includegraphics[width=0.48\textwidth]{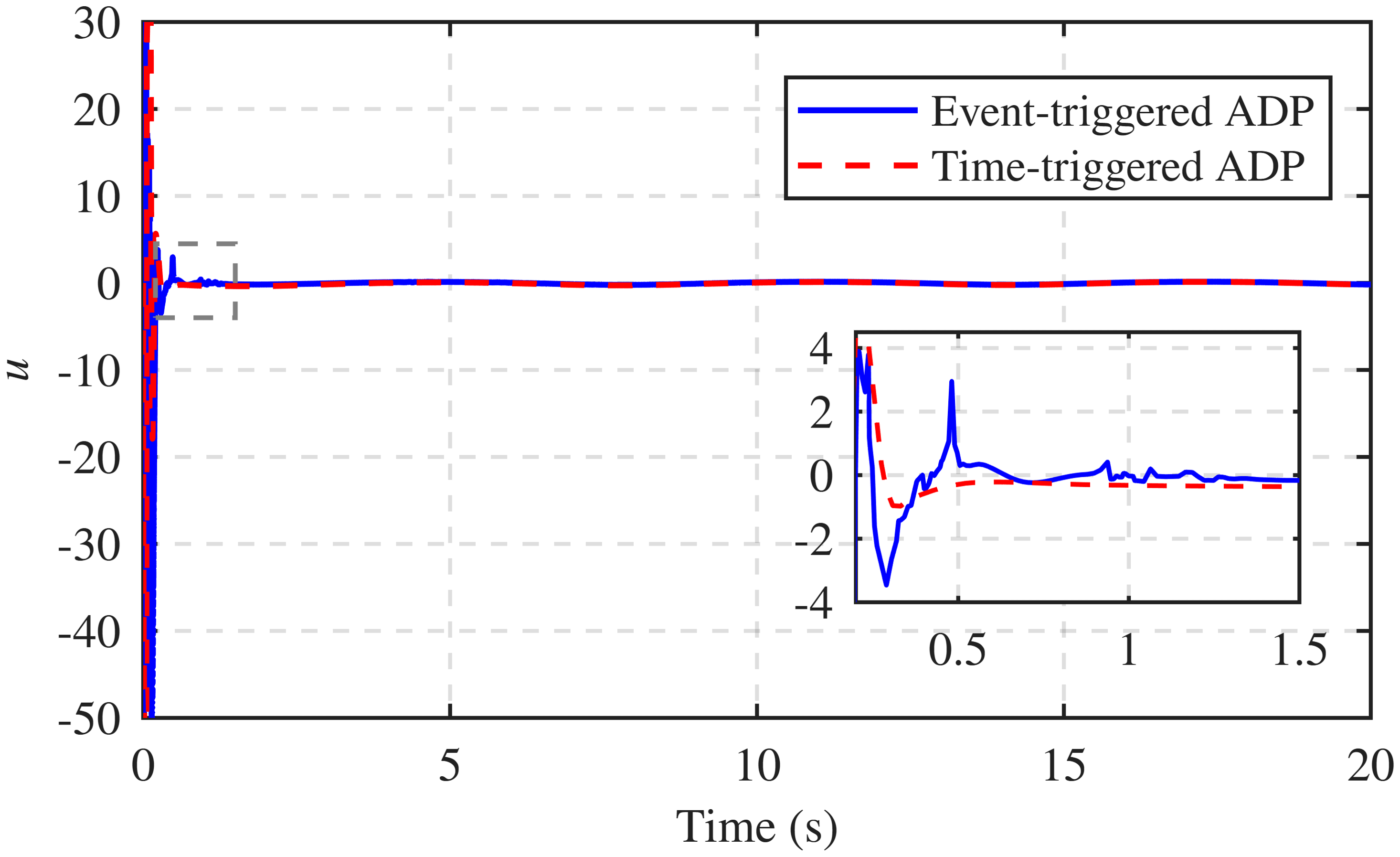}
\caption{Control signal composite control framework compared with those generated by the periodic \gls{ADP} strategy.}
\label{fig:ET_TT_u}
\end{figure}

As illustrated in Fig.~\ref{fig:ET_TT_u}, the control input produced by the composite controller stays well behaved and within bounds, even when roughly $71\%$ of its possible update instances are omitted. In contrast to the periodically updated \gls{ADP} (red dashed), \gls{ETM} greatly reduces the number of control updates while maintaining comparable transient response and steady-state performance.


Fig. \ref{fig:ET_dis} shows the distribution of triggering instants over the simulation horizon. Each cross mark represents an execution of the actor–critic update when the event-triggering condition is satisfied.
\begin{figure}[t]
\centering
\includegraphics[width=0.48\textwidth]{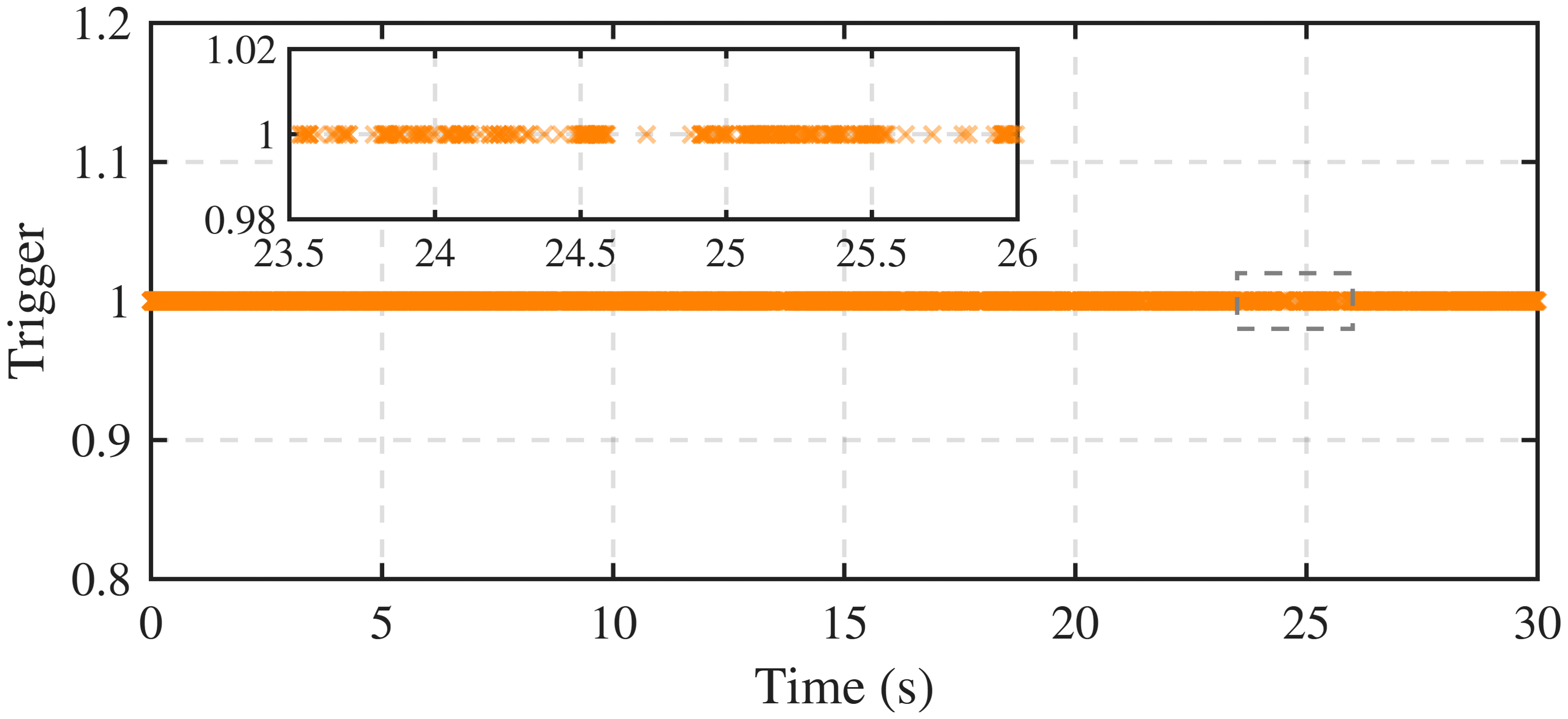}
\caption{Event trigger distribution}
\label{fig:ET_dis}
\end{figure}

\subsection{Example 2}

In this example, the proposed control method is implemented on an inverted pendulum system subject to both internal and external nonlinear disturbances. 
The system dynamics are formulated as follows:
\begin{equation}
\left\{
\begin{aligned}
\dot z &= 
\underbrace{0.5\,\eta(t)\,z}_{f_z(z,\eta)},
\\[6pt]
\dot x_1 &=
\underbrace{x_2}_{f_{x_1}(x_2)},
\\[6pt]
\dot x_2 &= 
\underbrace{-\frac{g}{l}\sin(x_1) - \frac{b}{m l^{2}}x_2}_{f_0(x_1,x_2)}
\\[4pt]
&\quad+
\underbrace{5e^{-0.3t} + 0.5\sin(x_1) + 0.5 z}_{\Delta f(x_1,z,t)}
+
\underbrace{\frac{1}{m l^{2}}}_{g_0}\,u,
\\[6pt]
y &= x_1.
\end{aligned}
\right.
\label{eq:inverted_pendulum_dynamics}
\end{equation}


In this example, we construct the \gls{BE} over a predefined rectangular exploration set: $
\mathcal{X}_E
=
[-2,\,2]_{0.5}
\times
[-5,\,5]_{1.0}.
$

\noindent The observer gain is selected as $L = \begin{bmatrix} 2 & 2 & 1 \end{bmatrix}^{\mathrm{T}}, $ The small positive constant used in the \gls{ESO} update is set to $\varepsilon = 0.03$. The saturation bounds for the \gls{ESO} states are configured as $M_1 = 1, M_2 = 1, M_3 = 3$, and the saturation bounds for the nominal functions are selected as $M_f = M_g = 7$. 

\noindent The nonlinear pendulum employed in the simulation is characterized by the parameters $ m=0.8 \mathrm{~kg}, l=1.2 \mathrm{~m}, b=0.2, g=9.81 \mathrm{~m} / \mathrm{s}^2. $ Here $m$ denotes the pendulum mass, $l$ the rod length, $b$ the viscous friction coefficient, and $g$ is the gravitational acceleration.

\subsubsection{Performance Analysis}

Fig.~\ref{fig:state_pole} shows that the \gls{ESO} reconstructs all system states, including the aggregated disturbance term, with high accuracy. Even during fast transients and highly nonlinear phases, the estimated trajectories $\hat{x}_i$ remain tightly aligned with the true states, showing only minimal deviation. The inset plots further demonstrate that the \gls{ESO} can track fast dynamics and suppress disturbances with rapid convergence.

\begin{figure}[t]
\centering
\includegraphics[width=0.45\textwidth]{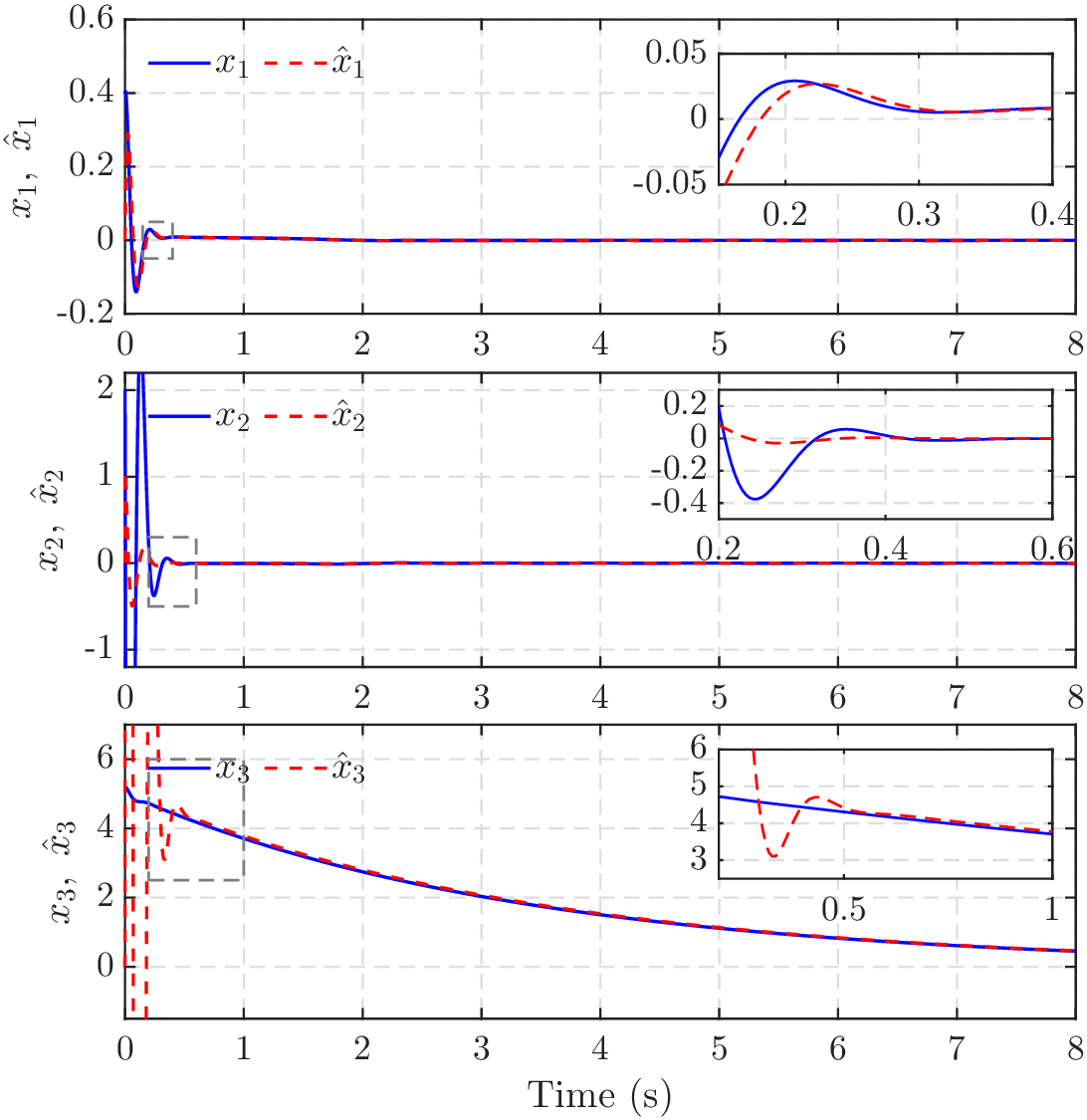}
\caption{System state trajectories and their \gls{ESO} estimates.}
\label{fig:state_pole}
\end{figure}


From a theoretical perspective, Fig.~\ref{fig:ET_TT_states_pole} illustrates the main benefit of the proposed composite control framework over its time-triggered counterpart. \gls{ETM} updates only when the state deviation crosses a prescribed threshold, allowing the learning and control actions to respond only to relevant changes in the dynamics. As a result, the \gls{ET}-\gls{ADP} achieves faster and more structured convergence, most notably in the $x_2$ response, which uses control and learning updates more efficiently than the uniformly sampled, time-triggered scheme. In this example, our mechanism achieved a $56\%$ reduction in computational cost.

\begin{figure}[t]
\centering
\includegraphics[width=0.48\textwidth]{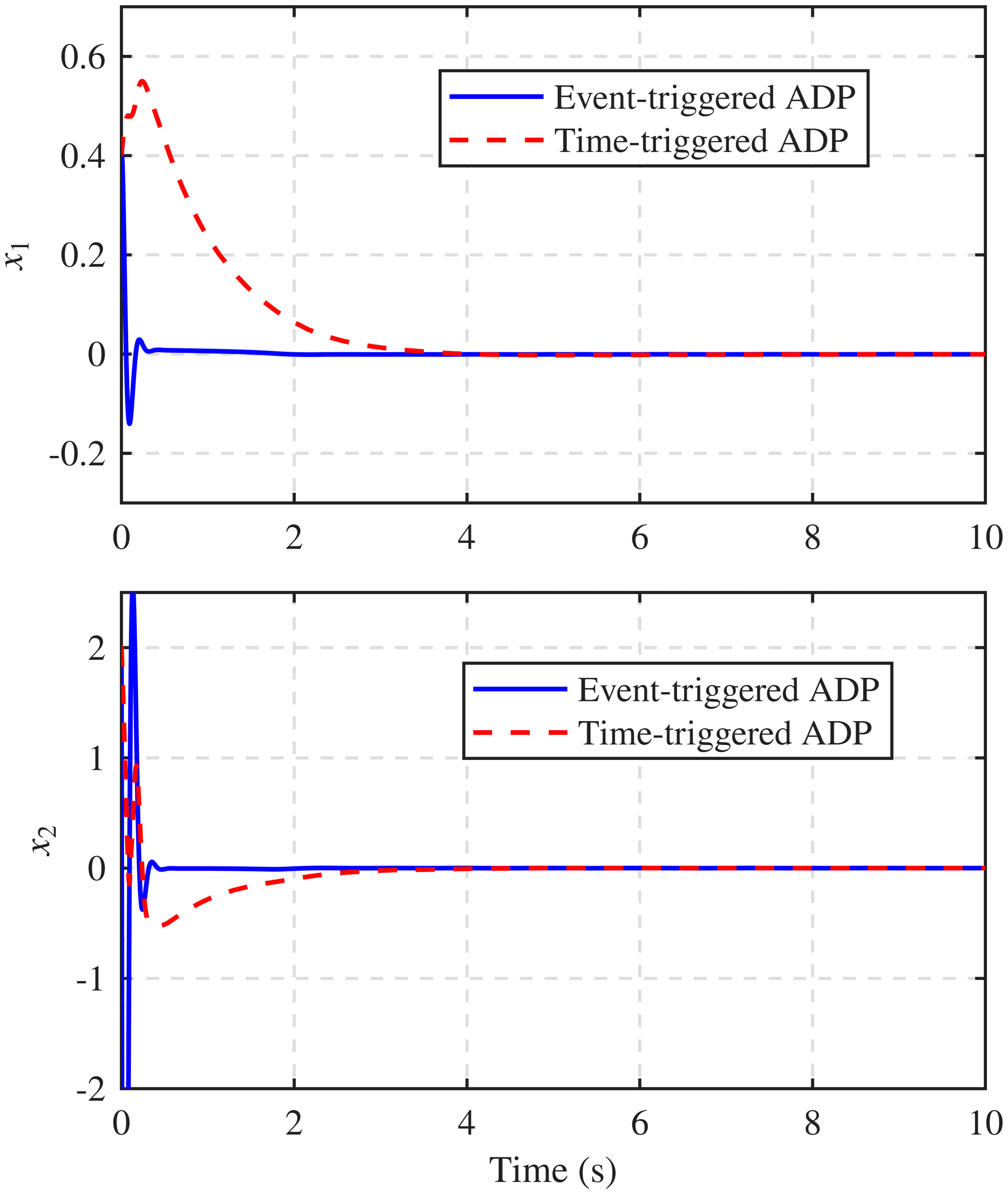}
\caption{Evolution of the system states from composite control framework and standard \gls{ADP} method}
\label{fig:ET_TT_states_pole}
\end{figure}

As depicted in Fig.~\ref{fig:ET_TT_u_pole}, the proposed framework controller delivers a quick corrective action during the initial transient phase, which is expected for stabilizing the nonlinear pendulum. After the system reaches steady state, the control input rapidly decays and remains near zero, reflecting stable closed-loop behavior and low steady-state effort.

Fig.~\ref{fig:ET_TT_u_pole} further shows that the \gls{ET}-\gls{ADP} controller exhibits smaller amplitude variations and a faster convergence compared with the conventional time-triggered \gls{ADP}. The time-triggered controller, by contrast, produces larger alternating control swings and wide input excursions, which are consistent with the overshoot and oscillatory behavior observed in Fig.~\ref{fig:ET_TT_states_pole}. These results highlight the benefits of \gls{ETM}: by updating only when necessary, it avoids excessive corrective actions, suppresses destabilizing oscillations, and promotes a more stable, energy-efficient control behavior.

\begin{figure}[H]
\centering
\includegraphics[width=0.48\textwidth]{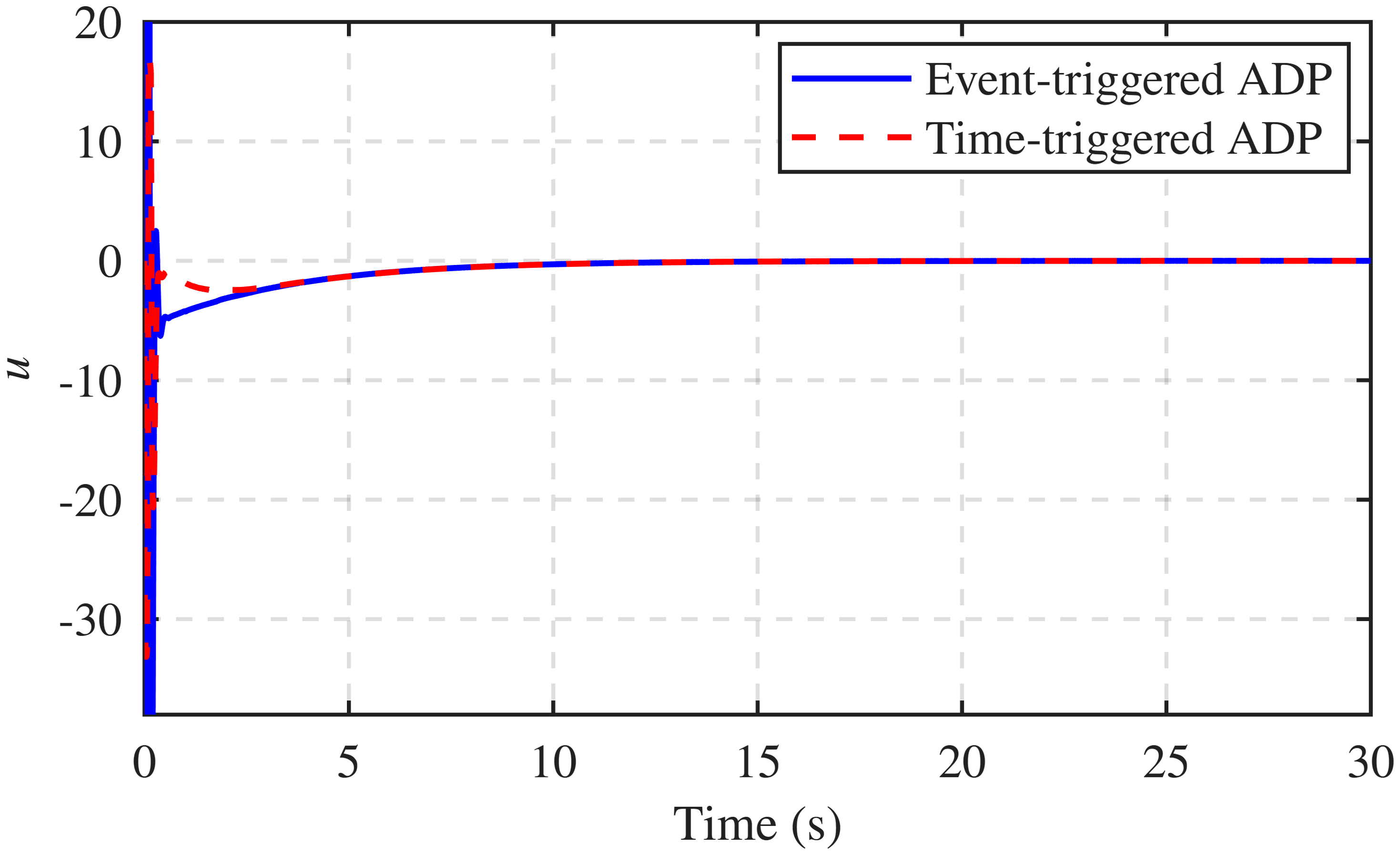}
\caption{Control signals generated by the proposed composite control framework, compared with those obtained from the periodic \gls{ADP} strategy.}
\label{fig:ET_TT_u_pole}
\end{figure}

\begin{figure}[ht!]
\centering
\includegraphics[width=0.48\textwidth]{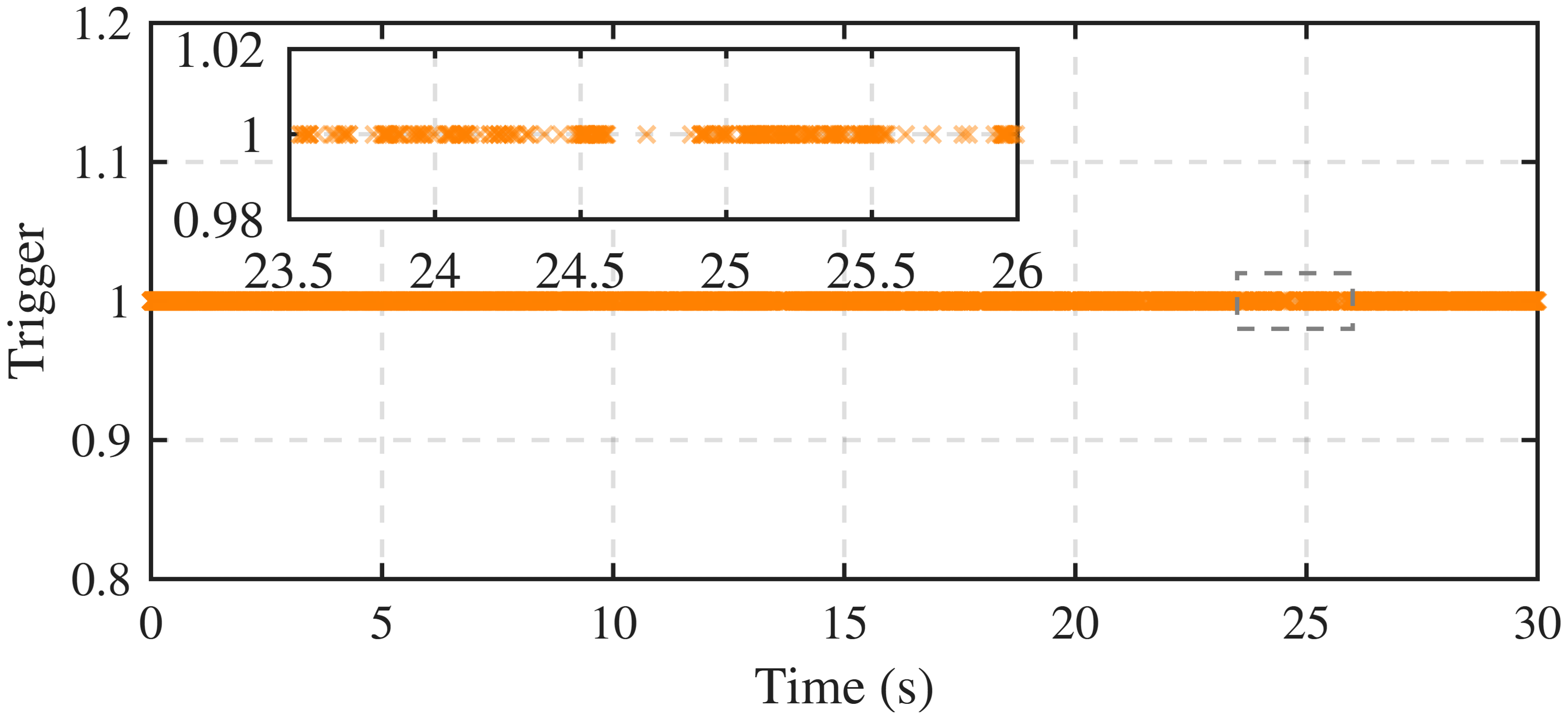}
\caption{Event trigger distribution}
\label{fig:ET_dis_2}
\end{figure}
Fig. \ref{fig:ET_dis_2} depicts the timing of the triggering events during the simulation. A cross indicates an instance where the event-triggering condition prompts an actor–critic update.

\section{Conclusion}
An \gls{ESO}-assisted \gls{ADP} architecture with an \gls{ETM} has been presented for uncertain nonlinear systems. In this framework, the \gls{ESO}-based compensation scheme provides real-time estimation and removal of lumped uncertainties, while the augmented state formulation embeds the tracking error and system dynamics into the optimal control design. The \gls{ADP} controller is then employed to approximate the optimal policy of the compensated subsystem through online learning. Simulation studies verify the framework's capability, showing that the controller maintains strong resistance to disturbances despite reduced update frequency enabled by the \gls{ETM}. The developed methodology will be further expanded to handle multi-input-multi-output configurations in future investigations. Also, a complete hybrid-system stability proof that explicitly incorporates the triggering error will be developed in the extended journal version.

\bibliography{adp_et}

\end{document}